\documentclass{amsart}
\theoremstyle{plain}
\newtheorem{thm}{Theorem}

\theoremstyle{definition}
\newtheorem{deff}{Definition}
\theoremstyle{remark}
\newtheorem{rmk}{Remark} 

\numberwithin{equation}{section}

\begin{document}                                                 
    \title[Homogenization and continuum percolation]{Homogenization and continuum percolation}                                 
    \author[Dimitris Kontogiannis]{Dimitris Kontogiannis}                                
    \address{Iowa State University}                                    
    \email{dkontog@iastate.edu}                                      
    \date{}                                       
    \keywords{Homogenization, random media, continuum percolation}                                   
    \subjclass[2000]{Primary: 35; Secondary:76M50}                                  
    \begin{abstract}We propose continuum percolation theory to study homogenization problems of elliptic equations. In particular, let $g:\mathbb{R}\rightarrow [0,1]$ be a connectivity function that connects two points $x_{1}, x_{2}\in\mathbb{R}^{n}$ with probability $g(|x_{1}-x_{2}|)$, where $|\cdot|$ denotes the Euclidean distance. Such functions have been introduced in continuum percolation theory. We involve the connectivity function to study homogenization problems in random media. Our aim is to improve and extend similar results that have been obtained for periodic domains.  
\end{abstract}                
    \maketitle  
\section{Introduction}
Homogenization theory is related to the asymptotic behavior of partial differential equations describing physical phenomena in heterogeneous materials. In particular, we look for the effective (homogenized) equations that describe the characteristics of the inhomogeneous medium as the length of a small parameter $\varepsilon$ tends to zero. This small parameter is the length of the heterogeneity.

Some methods to approach such problems are, among others, the two scale expansion, the two scale convergence, the $\Gamma-$ convergence. A summary of homogenization techniques can be found in \cite{Horn97}. Relevant problems have been studied in \cite{Allai90}, \cite{Allai92},
\cite{Cior82}, \cite{Cior99}, \cite{BoSu07}.

In this note, we obtain averaged equations over randomly perforated domains. For this purpose we use the connectivity function as defined in percolation theory \cite{Mees96}. More specifically, we consider the boundary value problem

\begin{equation}
\begin{array}{l}
	\Delta u^{\varepsilon}-\lambda u^{\varepsilon}=f , x\in G^{\varepsilon}(\omega)\cap D \\
  \displaystyle u^{\varepsilon}=0 , x\in\partial G^{\varepsilon}(\omega) \\
\end{array}
\label{eq:1}
\end{equation}
where the domain $G^{\varepsilon}(\omega)$ is a randomly perforated domain. We show that as the parameter $\varepsilon$ tends to zero, $u^{\varepsilon}$ converges to the solution of the problem
\begin{equation}
\begin{array}{l}
	\Delta u-(\lambda+c) u=f , x\in  D \\
  \displaystyle u=0 , x\in\partial D \\
\end{array}
\label{eq:1}
\end{equation}
The parameter $c$ that appears in $(1.2)$ is the extra term for the capacity which has been previously obtained for domains of periodic structure \cite{Cior82} as well as domains perforated by balls of random radius centered on the discrete $\mathbb{Z}^{n}$ lattice \cite{Caffa09}. Our approach is to provide a general model of randomly perforated domains so that $(1.2)$ can be obtained without periodicity assumptions. The main feature of our model is the connectivity function which has been introduced in continuum percolation theory\cite{Mees96}. This function connects two points of a point process (Poisson process for instance) with probability which depends on their distance. Thus, the main idea behind our model is 'as the density of points increases, points are connected only if their distance is of order $\varepsilon$'. With this restriction, the model of random structures and shapes is based on the distance of the points. For example, we can think of thin structures such as tubules across the line segments or balls centered on the points with sufficiently small radius. To establish the homogenized equation, we use the ergodic theorem of stochastic integral functionals proved in \cite{DalMod86} together with the ergodic properties of our model. 

We mention that in \cite{Cior82}, the 'strange term' $c$ appears in the equation when the 'holes' have a critical size and examples for periodic domains are included. In some cases, the distribution $c$ is given explicitly and the fundamental solution of the Laplacian on the annulus is an important tool.  Our approach covers a wide range of non periodic perforated domains and critical values on the size of the structures are not under consideration in this work. It would be of interest though to find examples of random domains of this generality for which critical values can be approximated or computed. 

The paper is organized as follows: in section $2$, we give a small review of results in stochastic homogenization, some of which we extend and improve in this note. In section $3$ we include a brief review of discrete and continuum percolation theory and the construction of random models that we will use in homogenization problems. In section $4$, we present our main homogenization results. Our work is an improvement of the homogenization method introduced by E.Khruslov(see for instance \cite{Khrus05}). A key of the method is the definition of 'mesoscopic' characteristics, measureson a scale $h$ with $\varepsilon<< h<<\text{ diam }D$ and has been used in a variety of similar problems.

\section{Ergodic theory and stochastic homogenization}

Stochastic homogenization is based on the most generalized notion of periodicity, the stationarity. If, in addition, we have independence at large distances, we talk about ergodicity. Ergodic theory is related to the study of dynamical systems with an invariant measure. 

Let $(\Omega,F,\mu, T)$ be a measure preserving dynamical system defined on a probability space with the following structure: $F$ is the $\sigma-$algebra on $\Omega$, ${\mu}$ is the probability measure, $T$ is the measure preserving transformation  such that for any $A\in F$, $\mu (T^{-1}(A))=\mu (A)$. The system $(\Omega,F,\mu, T)$ is ergodic if the $\sigma-$algebra of $T-$invariant events is trivial, that is, it occurs with probability zero or one. To see the importance of the theory on averaging problems, we state (among many versions) a subadditive ergodic theorem.

A function $\mu:A\rightarrow\mathbb{R}$ is called subadditive if for every finite and disjoint family $(A_{i})_{i\in I}$ with $\displaystyle|A\setminus\cup_{i\in I}A_{i}|=0$, 

\[\mu(A)\leq\sum_{i}\mu(A_{i})\] 
We say that $\mu$ is dominated if $0\leq\mu(A)\leq C|A|$ for all sets $A$. Consider now the family of dominated, subadditive functions and the group of translations $(\tau_{z}\mu)(A=\mu(\tau_{z}A)$, where $\tau_{z}A=\{x\in \mathbb{R}^{n}:x-z\in A\}$.

\begin{thm}(Ergodic)(see \cite{Akcog81}, \cite{DalMod86}): Let $\mu:\Omega\rightarrow \mathbb{R}^{n}$ be a subadditive process, periodic in law, in the sense that $\mu(\cdot)$ and $\tau_{z}\mu(\cdot)$ have the same distribution for every $z\in\mathbb{Z}^{n}$. Then, there exists measurable function $\phi:\Omega\rightarrow R$ and a subset $\Omega'\subset\Omega$ of full measure such that
\[\lim_{t\rightarrow\infty}\frac{\mu(\omega)(tQ)}{|tQ|}=\phi (\omega)\] exists a.e. $\omega\in\Omega'$ and for every cube $Q\subset \mathbb{R}^{n}$. Furthermore, if $\mu$ is ergodic then $\phi$ is constant. 
\end{thm}

 We refer a paper from Dal Maso-Modica \cite{DalMod86} for the use of ergodic theory in the calculus of variations, according to the following setting:
 
 Define the translation operator $\tau_{z}$ that acts through the following relations:\\ $\tau_{z}u(x)=u(x-z)$, and $\tau_{z}A=\{x\in \mathbb{R}^{n}:x-z\in A\}$, and the homothety operator $\displaystyle(\rho_{\varepsilon}J)(u,A)=\varepsilon^{n}J(\rho_{\varepsilon} u,\rho_{\varepsilon}A)$ where
 $(\rho_{\varepsilon}u)(x)=\frac{1}{\varepsilon}u(\varepsilon x)$, $\displaystyle(\rho_{\varepsilon}A)=\{x\in\mathbb{R}^{n}:\varepsilon x\in A\}$. A stochastic homogenization process is a family of random variables $(J_{\varepsilon})_{\varepsilon>0}$ on a probability space $(\Omega,F,P)$ that has the same distribution law with the random functionals given by 
 $\displaystyle[(\rho_{\varepsilon}J)(\omega)(u,A)]$ for $u\in W^{1,p}(A)$. This means
 
 \[ P\{\omega\in\Omega:F_{\varepsilon}(\omega)\in S\}=P\{\omega\in\Omega:\rho_{\varepsilon}J(\omega)\in S\}\] for any open set $S$.
 
 We say that the random functional $J$ is stochastically periodic, that is, $J$ has the same distribution law as the random functional $\displaystyle (\tau_{z}J)(\omega)(u,A)=J(\tau_{z}u,\tau_{z}A)$.
 
\begin{thm}
Let $\displaystyle J(\omega)(u,A)=\int_{A}f(x,\nabla u)dx$ , where $f$ satisfies standard growth conditions: $k|p|^{2}\leq f(x,p)\leq K|p|^{2} $ for some positive constants $k,K$. Denote the minimizer of $F$ by $\displaystyle m(J,u_{0},A)=\min_{u}\{J(u,A):u-u_{0}\in W_{0}^{1,2}(A)\}$.
\\ If $J$ is a random integral functional  and if $J$ and $\tau_{z}J=J(\tau_{z}u,\tau_{z}A)$ have the same distribution law, then the limit \[\lim_{t\rightarrow\infty}\frac{m(J(\omega),u_{0},Q_{t})}{|Q_{t}|}\] exists. If in addition $J$ is ergodic, the limit is constant.
\end{thm}

The proof of this theorem is based on the ergodic theorem \cite{Akcog81}, since the function $\mu(\omega)(A)=m(J(\omega),l_{p},A)$ is dominated and subadditive. 

Related works to stochastic homogenization include \cite{Bensou78},\cite{PapaVa}, \cite{Bourg94}. We also mention the work of Cafarelli-Mellet \cite{Caffa09} in which the authors extended the results of Cioranescu-Murat \cite{Cior82} in the case that the obstacle problem is considered in a domain perforated by balls of random radius centered at the $\mathbb{Z}^{n}$ lattice. As they showed, depending on the capacity of the holes, we still have an additional term that appears in the averaged equations. 

\section{Percolation theory and Random modelling}\label{S:Ds}
Percolation theory deals with the behavior of connected elements in random graphs. A cluster is a simply connected group of elements. Percolation phenomena arise in transport, porous media, spread of deseases, conductivity problems, sea ice etc. The theory was introduced by Broadpent and Hammersley in 1957, when they considered the problem of fluid flow through a porous medium formed by channels, keeping in mind that some of the channels may be blocked.  

In the discrete version of percolation, we consider the $\mathbb{Z}^{n}$ lattice and for $p\in[0,1]$ we connect the point $x\in\mathbb{Z}^{n}$ to each of its $2n$ nearest neighbors with probability $p$, independently of the other points.  We describe distances in the lattice in the following sense: two vertices $x,y\in\mathbb{Z}^{n}$ are neighbors if $|x-y|=1$. We can also define boundary $\partial Q_{n}$ of the sets $Q_{n}=[1,n]^{d}\cap\mathbb{Z}^{n}$, where $\partial Q_{n}=\{y\in Z^{d}:\exists x\in Q_{n}:|x-y|=1\}$. 

According to this setting, each pair of neighbours has an edge between them with probability $p$. The edge is also called a bond. A path is a finite or infinite alternating sequence $(z_{1},e_{1},z_{2},e_{2},..)$ of vertices $z_{i}$ and bonds $e_{i}$ such that $z_{i}\neq z_{j}$ and $e_{i}\neq e_{j}$ for $i\neq j$. Two vertices are connected if there is a finite open path from one to the other. An open cluster is a set of connected vertices that is maximal with respect to this property. It can be either finite or infinite.

 The main question that arises is if there exists a critical threshold of $p$ at which an infinite cluster occurs. In many cases, like the two-dimensional lattices, the critical value $p_{c}$ can be computed explicitly. This critical probability is an increasing function of $p$. 
 
\section{Continuum Percolation}

 A more general situation appears in the models of continuum percolation, where the integer lattice is replaced by a random set of points in $\mathbb{R}^{n}$. The random positions are usually formed by the realization of a point process. Point processes are important models and have been used in a variety of problems such as environmental modeling, air pollution, weather radar images, traffic networks, statistical mechanics etc. For a complete account of the theory, please see \cite{Mees96}.
 
 A point process is thought as a random set of points in the space. In particular we give the following definition:
 \begin{deff} Let $B$ be a Borel subset of $\mathbb{R}^{n}$ and denote by $N$ the set of all counting measures in $\mathbb{R}^{n}$. Let $\psi\in N$ be a counting measure, i.e. a measure which is 1 on each point $x\in B$. Then $N$ is identified as the set of all such configurations of points in $\mathbb{R}^{n}$ without limit points. According to this setting, $\psi(A)=$ random number of points in $A$, for any set $A\subset \mathbb{R}^{n}$. Then, a point process is defined as a measurable map $X$ from a probability space $(\Omega_{1}, F, P)$ into $(N,M,P)$, where $M$ is the correspoding $\sigma-$algebra.
 \end{deff}
 
 The periodicity of the lattice-type structures is replaced with the assumption that the point process $X(\cdot)$ is stationary:
\\ If $T_{\alpha}$ is the translation in $\mathbb{R}^{n}$ by a vector $\alpha$, $T_{\alpha}(\beta)=\beta+\alpha$, $\forall\beta\in \mathbb{R}^{n}$, then $T_{\alpha}$ induces a transformation $S_{\alpha}:N\rightarrow N$ through the operation $(S_{\alpha}\psi)(A)=\psi(T_{\alpha}^{-1}A)$ $\forall A\in B^{n}$ and similar operation holds for the set-measures.
\begin{deff} The point process $X$ is stationary if its distribution is $S_{\alpha}-$invariant for any $\alpha\in \mathbb{R}^{n}$.\end{deff}

Two common models in the theory of continuum percolation are the Boolean model and the random connection model. Both models are based on occurences of Poisson processes. From this point, we assume that the point process $X$ is a Poisson process of density $\lambda$:
\\(i) for any collection of mutually disjoint sets $A_{1},A_{2},...,A_{k}$,\\ the random variables $X(A_{1}),X(A_{2}),...,X(A_{k})$ are mutually independent, and
\\(ii) for any bounded set $A\subset B^{n}$ and any non-negative integer $k\geq 0$, $\displaystyle P(X(A)=k)=e^{-\lambda l(A)}\frac{\lambda^{k}(l(A))^{k}}{k!}$, where $l(\cdot)$ is the Lebesgue measure.
\subsection{The Boolean model}

The Boolean model is driven by a Poisson process $X$ and each point of the process is the center of a ball of random radius. The region in the space that is covered by at least one ball is called the occupied region and its complement is the vacant region. To be able to have shifting properties, we construct the models as follows. The Poisson process is defined in a probability space $\left(\Omega_{1},F_{1},P_{1}\right )$ and we consider a second space $\displaystyle\Omega_{2}=\prod_{\psi\in N}\prod_{z\in Z^{n}}[0,\infty)$ equipped with the probability measure $\mu$ on $[0,\infty)$. Setting $\Omega=\Omega_{1}\times\Omega_{2}$ with product measure $P=P_{1}\times P_{2}$, the Boolean model is defined as the map $(\omega_{1},\omega_{2})\rightarrow\left( X(\omega_{1}),\omega_{2}\right)$ from $\Omega$ into $N\times\Omega_{2}$. 

According to this construction, the radii of the balls are independent of the point process and we obtain the shifting properties that ergodic theory requires.

We use the notation of [13] and we denote by $(X,\rho,\lambda)$ the Boolean model obtained from a Poisson process $X$ of density $\lambda$ and radius random variable $\rho$.  

\subsection{The Random connection model}

As in Boolean models, the Poisson process is the first characteristic of the model and it assigns randomly points in the space. The second characteristic of the model is the connection function, which plays an essential role to the model and the homogenization process as we will see later. A connection function $g:\mathbb{R}^{+}\rightarrow [0,1]$ connects two points $x_{1}, x_{2}\in X$ with probability $g(|x_{1}-x_{2}|)$, where $|\cdot|$ denotes the Euclidean distance. Depending on the construction that we need, we can choose $g$ with specific characteristics. For example, we may assume that $g$ is decreasing with respect to the distance. Such models are defined in product spaces as before. We denote them by $\left(X,g,\lambda\right)$.
\begin{deff} Two points $x,y$ of the process are connected if there is a sequence of points $\{x_{0}=x,x_{1},x_{2},....,x_{n}=y\}$  such that each pair of points $x_{i},x_{i+1}$ are the endpoints of a line segment (edge) $\{x_{i},x_{i+1}\}$ for all $i=0,....,n-1$. As in the discrete percolation, a component is a set of points such that any two points of this set are connected and the set is maximal with respect to this property. \end{deff}

\subsection{Ergodic properties of point processes}

 The one dimensional ergodic theorem is stated as follows (see for instance \cite{Mees96}):

\begin{thm}
Let $(\Omega,F,\mu, T)$ be a measure preserving dynamical system and let $f$ be $\mu-$integrable function on $\Omega$. Then, $\displaystyle\frac{1}{n}\sum_{i=0}^{n-1}f(T^{i}(\omega))\rightarrow E(f|I)(\omega)$, as $n\rightarrow\infty$ a.s. where $I$ is the $\sigma-$algebra of $T-$ invariant sets.
\end{thm}

For point processes, we identify  any element $\omega\in\Omega$ with a counting measure. Then the shift by distance $t$, $T_{t}$, induces a transformation through $(S_{t}\omega)(A)=\omega(T_{t}^{-1}A)$ for all measurable sets $A\in\mathbb{R}$.

\begin{deff}
A stationary point process is ergodic if the group $\{S_{x}:x\in\mathbb{R}^{n}\}$ acts ergodically on $(\Omega,F,\mu)$.
\end{deff}

An important aspect of continuum percolation models is that the ergodic properties of the process $X$ are carried over the two models. The following results are well known \cite{Mees96}:

\begin{thm}: A Poisson point process is stationary ergodic.\end{thm}
\begin{thm} Suppose that the point process $X$ is ergodic. Then the random connection model $(X,g,\lambda)$ and the Boolean model $(X,\rho,\lambda)$ are also ergodic.\end{thm}

Let $(\Omega_{1}, F_{1},P_{1})$ be a probability space associated with the Poisson process $X$ and let $\omega\in\Omega_{1}$ be a point configuration on $\mathbb{R}^{n}$ that is assumed to be locally finite and countable. This means that we have a finite number of points hits every compact set $K\subset\mathbb{R}^{n}$ almost surely: \[ P(\omega\in\Omega:\psi(K)<\infty\text{ for all compact }K\subset\mathbb{R}^{n})=1\]

\section{Construction of random domains for homogenization}
\subsection{Random connection models}
We start using the random connection model $(X,g,\lambda)$ in the following way: 
 
Suppose that $\omega$ is a given realization for $X$ which is locally finite and let $x_{i}\in X$ be a given point of this realization.

Consider the annulus $A=\{x\in\mathbb{R}^{n}: c_{1}\leq |x-x_{i}|\leq c_{2}\}$, where $c_{1},c_{2}$ are positive constants with $c_{1}\leq c_{2}$.  

 We want to connect the point $x_{i}$ with all the points in $A$ that are given from $X$. For this purpose we choose the connection function
 \[g(|x-x_{i}|)=
	\begin{cases}
	1 & \text{if } c_{1}\leq|x-x_{i}|\leq c_{2} \\
  0 & \text{otherwise} 
  \end{cases}\]
  
 For a point $x_{j}\in A$, we denote by $l_{ij}(\omega)=l(x_{i},x_{j})$ the line segment with endpoints $x_{i},x_{j}$ and let $T_{c_{1}/2}(l_{ij})(\omega)$ the tube of radius $c_{1}/2$ surrounding $l_{ij}$. Let now $\displaystyle T(x_{i})(\omega)=\cup_{j}T_{c_{1}/2}(l_{ij})(\omega)$ and $\displaystyle F(\omega,c_{1}/2)= \cup_{i}T(x_{i})(\omega)$ for all points $x_{i}$ of the process.   
 
 Thus, the set $F(\omega,c_{1}/2)$ is the union of random tubes obtained from the given realization of the point process. Let $G(\omega,c_{1}/2)=\mathbb{R}^{n}\setminus F(\omega,c_{1}/2)$.  We define the indicator function
\[a(\omega,x)=1-min\{X_{F(\omega, c_{1}/2)},1\}\] which is zero in the union of tubes and one elsewhere.  

Let $D$ be an open, bounded domain of $\mathbb{R}^{n}$ and consider the random functional  $\displaystyle J(\omega)(u,D)=\int_{D}a(\omega,x)|\nabla u|^{2}dx=\int_{G(\omega)\cap D}|\nabla u|^{2}dx$ for $u\in W^{1,2}(D)$. This functional is periodic in law and independent at large distances, thus ergodic. 

 Furthermore let $\displaystyle(\rho_{\varepsilon}J)(u,A)=\varepsilon^{n}J(\rho_{\varepsilon} u,\rho_{\varepsilon}A)$ where
 $\displaystyle(\rho_{\varepsilon}u)(x)=\frac{1}{\varepsilon}u(\varepsilon x)$, $\displaystyle(\rho_{\varepsilon}A)=\{x\in\mathbb{R}^{n}:\varepsilon x\in A\}$.
Then the family 
\begin{displaymath}
J^{\varepsilon}(u,D)=\rho_{\varepsilon}J(u,D)=\int_{\varepsilon^{n} G(\omega)\cap D}|\nabla u|^{2}dx
\end{displaymath}
satisfies the assumptions of theorem $2$. Note that the $\rho_{\varepsilon}-$ homothetic functional is the functional obtained if we scale by $\varepsilon$ the distance between the connected points of the set $F(\omega,c_{1}/2)$ that corresponds to the union of tubes $\varepsilon F=F(\varepsilon\omega,\varepsilon c_{1}/2)$, where $\varepsilon\omega$ maps to the point measure whose support is $\{\varepsilon x_{i}\}$ and $\{x_{i}\}$ is the support of $X(\omega)$. Note that the scaling properties of this model are the same (in terms of distribution) with the model that we have if we choose 
\[g_{\varepsilon}(|x-x_{i}|)=
	\begin{cases}
	1 & \text{if }c_{1}\varepsilon\leq |x-x_{i}|\leq c_{2}\varepsilon \\
  0 & \text{otherwise} \\
  \end{cases}\]
with density function $\lambda/\varepsilon$. Let us define $F^{\varepsilon}(\omega)=\varepsilon F=F(\varepsilon\omega,\varepsilon c_{1}/2)$ and $G^{\varepsilon}(\omega)=\varepsilon G(\omega)=\mathbb{R}^{n}\setminus F^{\varepsilon}(\omega)$.  According to this model, only points of $\varepsilon-$distance are connected.

\begin{rmk}
The radius of the tubes need not necessarily be constant. We may consider the tubes $T_{\rho_{\varepsilon}(\omega)}(l(x_{i},x_{j}))=\{x\in\Omega: d(x,l(x_{i},x_{j}))\leq\rho_{\varepsilon} (x,\omega)\}$, where $\rho_{\varepsilon}(x,\omega)$ is continuous function on x, stationary ergodic and for some positive constants $\rho_{1,\varepsilon},\rho_{2,\varepsilon}$, $\rho_{1,\varepsilon}\leq\rho_{\varepsilon} (x,\omega)\leq\rho_{2,\varepsilon}$. This is due to the fact that the product of two ergodic processes is also ergodic.
\end{rmk} 
\subsection{Boolean models}

We take again a probability space $(\Omega_{1}, F_{1},P_{1})$ associated with the Poisson process $X$ and let $\omega\in\Omega$ be a given point configuration on $\mathbb{R}^{n}$. We consider a second space $\displaystyle\Omega_{2}=\prod_{\psi\in N}\prod_{z\in Z^{n}}[0,\infty)$ equipped with the probability measure $\mu$ on $[0,\infty)$ for the sequence of independent, identically distributed random radii. Let us denote by $\bar{\omega}=(\omega,\bar{r})$ the realization of this Boolean model in $\Omega=\Omega_{1}\times\Omega_{2}$, where $\bar{r}=(r_{1},r_{2},...)$. 

Let $F(\omega)$ be the union of random spheres obtained from the given realization of the point process. Let $G(\omega)=\mathbb{R}^{n}\setminus F(\omega)$. We define the indicator function 
 
$\displaystyle a(\omega,x)=1-min\{X_{F(\omega)},1\}$ which is zero in the union of spheres and one elsewhere. 

Let $D$ be an open, bounded domain of $\mathbb{R}^{n}$ and consider the random functional  $\displaystyle J(\omega)(u,D)=\int_{D}a(\omega,x)|\nabla u|^{2}dx=\int_{G(\omega)\cap D}|\nabla u|^{2}dx$. This functional is periodic in law and independent at large distances, thus ergodic. 

 Then, the family $\rho_{\varepsilon}J(\omega)(u,D)=J^{\varepsilon}(u,D)$ satisfies the assumptions of theorem $2$.

A Boolean model gives, in general, a union of spheres which may intersect. One way to model non-intersecting spheres is to combine  the Boolean and the random connection model in the following way:
\\ Suppose the RCM $(X,g,\lambda)$ is applied on a bounded region of $\mathbb{R}^{n}$. We want to assign every endpoint of the line process as the center of a ball of random radius. For fixed $\varepsilon>0$ there is a set of points from the point process X. In our case, instead of constructing tubes, we let every point be the center of a ball with radius $\displaystyle\rho(\omega)\leq\min d(x_{i},x_{j})(\omega)$, where the minimum is taken over all the pairs of points $x$ of $X(\omega)$. Note that, without any affect to our proofs, we may assume that $\rho(\omega)$ is identically distributed random variable taking maximum value $\min d(x_{i},x_{j})(\omega)$. We consider for simplicity the first case. According to this construction, we obtain a domain randomly perforated with balls of radius and with positive minimal distance to each other. Let us define \[\displaystyle F^{\varepsilon}(\omega)=\bigcup_{i\geq 1}B(\varepsilon\rho(\omega),\varepsilon x_{i})\cap D\] and $G^{\varepsilon}(\omega)\cap D=D\setminus F^{\varepsilon}(\omega)$. Note that $\text{mes}F^{\varepsilon}(\omega)$ tends to zero as $\varepsilon\rightarrow 0$.

\section{The Dirichlet problem}
\subsection{Boolean models and domains with fine-grained boundary}

Let $\omega\in\Omega_{1}$ be such that $X(\omega)$ is locally finite.

We consider the Dirichlet problem of the form
\begin{equation}
\begin{array}{l}
	\Delta u^{\varepsilon}-\lambda u^{\varepsilon}=f , x\in G^{\varepsilon}(\omega)\cap D \\
  \displaystyle u^{\varepsilon}=0 , x\in\partial G^{\varepsilon}(\omega) \\
\end{array}
\label{eq:1}
\end{equation}

for $u^{\varepsilon}\in W^{1,2}(G^{\varepsilon}(\omega))$, $f\in L^{2}(D)$. Here, $G^{\varepsilon}(\omega)$ is the domain perforated by  non-intersecting balls. 

To define capacity characteristics for the massiveness of $B_{i}^{\varepsilon}(\omega)$, let us consider the quantity 

\[\text{cap}(B)=\inf_{v}\int_{\mathbb{R}^{n}}|\nabla v|^{2}dx\]
for closed, bounded sets $B$ in $\mathbb{R}^{n}$, over all functions $v\in C_{0}^{\infty}(\mathbb{R}^{n})$ taking value $1$ in $B$. We consider only the case $n\geq 3$ so that  $\text{cap}(B)$ defines the Newton capacity. Note that the capacity is invariant with respect to translations and rotations. In addition, if $B_{\varepsilon}$ is the $\varepsilon-$ homothetic contraction of $B$, $\text{cap}(B_{\varepsilon})=\varepsilon^{n-2}\text{cap}(B)$. 

We also note that, clearly, as $\varepsilon\rightarrow 0$, the diameter of the balls tend to zero. Now, note that the limits $\displaystyle\lim_{\varepsilon\rightarrow 0}\sum_{(D)}\text{cap}(B_{i}^{\varepsilon}(\omega))=\text{cap}(G^{\varepsilon}(\omega))=C$ exists  due to the ergodic properties of the model and theorem $1$. The sum is taken over all balls strongly contained in $D$.

We state the homogenization theorem for the Dirichlet problem related to Boolean models.
\begin{thm} 
The family of solutions $u^{\varepsilon}$ of the Dirichlet problem $(3)$ (extended by zero in $\displaystyle\bigcup_{i=1}^{n(\varepsilon)}B_{i}^{\varepsilon}(\omega)$) converges in $L^{2}(D)$ to the solution $u$ of the boundary value problem

\begin{equation}
\begin{array}{l}
	\Delta u-(\lambda+C)u=f , x\in D \\
   u=0 , x\in\partial D \\
\end{array}
\label{eq:2}
\end{equation}
\end{thm}

The proof of this theorem can be found in \cite{Khrus05} under general assumptions which have been proved in our cases. 

\subsection{RCM and connected domains}
Let $F^{\varepsilon}(\omega)$ be the random set of channels that we constructed in section $4.2$. 
We assume that the Poisson process $X$ is locally finite in the sense that a finite number of points hits every compact set $K\subset\mathbb{R}^{n}$ almost surely. That is, $\displaystyle P(\omega\in\Omega:\psi(K)<\infty\text{ for all compact }K\subset\mathbb{R}^{n})=1$.

Let $G^{\varepsilon}(\omega)=D\setminus F^{\varepsilon}(\omega)$. We want to show that for sufficiently small $\varepsilon$, the volume of $G^{\varepsilon}(\omega)$ is strictly positive with probability $1$. That is, there is $\hat{\varepsilon}(\omega)>0$ such that
\begin{displaymath}  
P(\omega\in\Omega:|G^{\varepsilon}(\omega)|>0\text{ for all }\varepsilon<\hat{\varepsilon}(\omega))=1
\end{displaymath}

For simplicity and without loss of generality, suppose that $n=2$ and $D$ is a square of size $1/\varepsilon>0$. Consider a partition of $D$ into squares $D_{i}$, $i=1,..,1/\varepsilon^{2}$ of size $1$. Suppose that the point process $X$ is applied in $D$. Then, the probabilty of having zero points in a given $D_{i}$ is $e^{-\lambda}>0$. According to the law of large numbers, if $1_{i}(\omega)$ is the indicator function of the empty square $D_{i}$ of the partition,
\[P(\omega\in\Omega:\lim_{\varepsilon\rightarrow\infty}\frac{1}{1/\varepsilon^{2}}\sum_{i=1}^{\varepsilon^{2}}1_{i}(\omega)=e^{-\lambda})=1\]
which means that, in the limit, there are empty squares with probability one, as needed.
Since the points, within an $\varepsilon-$range, are connected, the set \[K^{\varepsilon}(\omega)=\bigcup_{\{i,j\}}T_{\rho_{\varepsilon}}(l_{ij})\bigcap\{\bigcup_{\{k,l\}\neq\{i,j\}} T_{\rho_{\varepsilon}}(l_{kl})\}\] (the intersection of the tubes) is non-empty, $K^{\varepsilon}(\omega)\subset F^{\varepsilon}(\omega)$.
 Clearly, as $\varepsilon\rightarrow 0$, $\text{mes}F^{\varepsilon}(\omega)\rightarrow 0$ and $G^{\varepsilon}(\omega)$ becomes denser in $\Omega$. One may consider two different homogenization problems: the first is when the elliptic equation is defined in $G^{\varepsilon}(\omega)$ for $u^{\varepsilon}\in W^{1,2}(G^{\varepsilon}(\omega))$ and the second is the problem of decreasing volume when $u^{\varepsilon}\in W^{1,2}(F^{\varepsilon}(\omega))$.

We consider the Dirichlet problem of the form
\begin{equation}
\begin{array}{l}
	 -\Delta u^{\varepsilon}+\lambda u^{\varepsilon}=f , x\in G^{\varepsilon}(\omega) \\
   u^{\varepsilon}=0 , x\in\partial G^{\varepsilon}(\omega) \\
\end{array}
\label{eq:3}
\end{equation}
for $u^{\varepsilon}\in W^{1,2}(G^{\varepsilon}(\omega))$, $f\in L^{2}(D)$ and as $\varepsilon\rightarrow 0$, $G^{\varepsilon}(\omega)$ is approximately $D$. In general, $F^{\varepsilon}(\omega)$ consists of connected components, but it is not necessarily a connected set.

Standard elliptic theory gives the existence of solutions. We extend the functions $u^{\varepsilon}$ by zero in $F^{\varepsilon}(\omega)$ and we keep the same notation for the extended sequence of functions. We denote by $Q_{h}^{x}$ the n-cube centered at $x$ of length $h$, $\text{diam}D>>h>>\varepsilon$, and we define the local capacity functional
\begin{equation}
\begin{array}{l}
\displaystyle\text{cap}(x,h,\varepsilon,\omega)=\inf_{u^{\varepsilon}}\int_{Q_{h}^{x}}|\nabla u^{\varepsilon}(y)|^{2}dy
\end{array}
\label{eq:4}
\end{equation} over all $u^{\varepsilon}\in W^{1,2}(Q_{h}^{x}):u^{\varepsilon}=0$ in $F^{\varepsilon}(\omega)$. Clearly, if $F^{\varepsilon_{1}}(\omega)\cap Q_{h}^{x}\subset F^{\varepsilon_{2}}(\omega)\cap Q_{h}^{x}$, then $\text{cap}(x,h,\varepsilon_{1},\omega)\leq\text{cap}(x,h,\varepsilon_{2},\omega)$. Thus, the capacity functional measures the massiveness of $F^{\varepsilon}(\omega)$ in $\Omega$. Note that, from our construction and due to theorem $(5)$, $F^{\varepsilon}(\omega)$ and its complement are periodic in law and ergodic in the sense that disjoint cubes have the same distribution.

Note that $\displaystyle\text{cap}(x,h,\varepsilon,\omega)=\inf_{u^{\varepsilon}}\int_{Q_{h}^{x}}a(x,\omega)|\nabla v^{\varepsilon}|^{2}dx$, where \[a(x,\omega)=1-\min\{1,X_{T_{\rho}(l_{ij})}\}\] which is $0$ in the union of random channels with endpoints the points of the Poisson process and $1$ elsewhere.

Hence, theorem $2$ is applicable over the class of functions in $W^{1,2}(D)$. Thus, the constant limit 
\[
 c=\lim_{h\rightarrow 0}\lim_{\varepsilon\rightarrow 0}\frac{\displaystyle \text{cap}(x,h,\varepsilon,\omega)}{h^{n}}
\]
exists almost all $\omega\in\Omega_{1}$. Hence, we can also assume that
\[
\limsup_{\varepsilon\rightarrow 0}\frac{\text{cap}(x,h,\varepsilon,\omega)}{h^{n}}< A\]
for all $x\in D$ with $A$ independent of $h$. Our main homogenization theorem is the following:
\begin{thm}
Let $F^{\varepsilon}(\omega)$ be the sequence of RCM domains constructed in section $5$ and $G^{\varepsilon}(\omega)$ be its complement set.
Let $u^{\varepsilon}\in W^{1,2}(G^{\varepsilon}(\omega))$ be the family of solutions of the boundary value problems $(6.3)$ extended by zero in $F^{\varepsilon}(\omega)$. Then, as $\varepsilon\rightarrow 0$, $u^{\varepsilon}$ converges in $L^{2}(D)$ to the limit $u\in W^{1,2}(D)$ which solves the boundary value problem
 \begin{equation}
\begin{array}{l}
	\Delta u-(\lambda+c)u=f , x\in D \\
  u=0 ,  x\in\partial D 
\end{array}
\label{eq:5}
\end{equation}
\end{thm}

\begin{proof}
Note that the (extended by zero) solution $u^{\varepsilon}$ of $(6.3)$ is the minimizer in $D$ of the functional $\displaystyle \Gamma^{\varepsilon}[u^{\varepsilon}]=\int_{G^{\varepsilon}(\omega)\cap D}|\nabla u^{\varepsilon}|^{2}+\lambda|u^{\varepsilon}|^{2}+2fu^{\varepsilon}dx=\int_{D}|\nabla u^{\varepsilon}|^{2}+\lambda |u^{\varepsilon}|^{2}+2fu^{\varepsilon}dx$ over the class of functions $u^{\varepsilon}\in W^{1,2}(G^{\varepsilon}(\omega))$.
\\ Thus, $\Gamma^{\varepsilon}[u^{\varepsilon}]\leq\Gamma^{\varepsilon}[0]=0$ which implies that 
\[\int_{G^{\varepsilon}(\omega)\cap D}|\nabla u^{\varepsilon}|^{2}+\lambda|u^{\varepsilon}|^{2}dx\leq 2\|u^{\varepsilon}\|_{L^{2}(\Omega)}\|f\|_{L^{2}(\Omega)}\]

 Using the Friedrich's inequality $\displaystyle\|u^{\varepsilon}\|_{L^{2}(D)}\leq C\|\nabla u^{\varepsilon}\|_{L^{2}(D)}$ we obtain that \[\|u^{\varepsilon}\|_{W^{1,2}(D)}\leq C\] where $C$ is independent of $\varepsilon$. Since $u^{\varepsilon}$ in bounded, it has a subsequence, still denoted by $u^{\varepsilon}$, that converges weakly in $W^{1,2}(D)$ and strongly in $L^{2}(D)$ to some function $u\in W^{1,2}(D)$. 
\\ To prove the theorem, it is enough to show that the limit $u$ is the minimizer of the functional $\displaystyle \bar{\Gamma}[u]=\int_{D}|\nabla u|^{2}+(\lambda+c)|u|^{2}+2fudx$.
\\\\ Step1: We first establish the inequality $\displaystyle\limsup_{\varepsilon\rightarrow 0}\Gamma^{\varepsilon}[u^{\varepsilon}]\leq \tilde{\Gamma}[w]$ for all $w\in W^{1,2}(D)$.
\\ For this purpose, we consider a partition of $D$ with cubes $Q^{\alpha}=Q(x^{\alpha},h)$ centered at $x^{\alpha}$ of size $h$, so that $\displaystyle\cup_{\alpha}Q(x^{\alpha},h)$ is a cover of $D$ and the points $x^{\alpha}$ form a periodic lattice of period $h-r$, $r$ to be chosen. A partition of unity $\{\phi_{\alpha}\}$ of $C^{2}$ functions subordinated to this covering is given as follows
\begin{enumerate}
\item $0\leq\phi_{\alpha}\leq 1$
\item $\phi_{\alpha}=0$ if $x\notin Q^{\alpha}$, $\phi_{\alpha}=1$ if $\displaystyle x\in Q^{\alpha}\setminus\cup_{\beta\neq\alpha}Q^{\beta}$
\item $\displaystyle\sum_{\alpha}\phi_{\alpha}(x)=1$, if $x\in D$
\item $|\nabla\phi_{\alpha}|\leq C/r$
\end{enumerate}
Let us denote by $v^{\alpha}=v^{\alpha(\varepsilon)}$ the minimizer of $\text{cap}(x,h,\varepsilon,\omega)$ in the cube centered at $x^{\alpha}$.
For $w\in C^{2}(D)$, compactly supported in $D$, define
\begin{equation}
\begin{array}{l}
 \displaystyle w_{h}^{\varepsilon}(x)=\sum_{\alpha=1}^{n(h)}w(x)v^{\alpha}(x)\phi_{\alpha}(x)=w(x)+\sum_{\alpha=1}^{n(h)}w(x)[v^{\alpha}(x)-1]\phi_{\alpha}(x)
  \end{array}
\label{eq:6}
\end{equation}
  so that $w_{h}^{\varepsilon}\in W^{1,2}(D)$ and $w_{h}^{\varepsilon}(x)=0$ in $F^{\varepsilon}(\omega)$. Thus, $\Gamma^{\varepsilon}[u^{\varepsilon}]\leq \Gamma^{\varepsilon}[w_{h}^{\varepsilon}]$. Under our assumptions, 
 \begin{equation}
\begin{array}{l}
\displaystyle\int_{Q_{h}^{\alpha}}|\nabla v^{\alpha}|^{2}dx\leq Ch^{n}
 \end{array}
\label{eq:7}
\end{equation}

We denote by $\displaystyle\hat{Q}_{h}^{\alpha}=Q_{h}^{\alpha}\setminus\cup_{\beta\neq\alpha}Q_{h}^{\beta} $ the concentric cube centered at $x^{\alpha}$ of size $\hat{h}=h-2r$.
Then, 
\begin{align*}&\int_{Q_{h}^{\alpha}\setminus\hat{Q}_{h}^{\alpha}}|\nabla v^{\alpha}|^{2}dx\\
&=\int_{Q_{h}^{\alpha}}|\nabla v^{\alpha}|^{2}dx-\int_{\hat{Q}_{h}^{\alpha}}|\nabla v^{\alpha}|^{2}dx+O(rh^{n-1})\\
&\leq\text{cap}(x,h,\varepsilon,\omega)-\text{cap}(x,\hat{h},\varepsilon,\omega)+O(rh^{n-1})\end{align*}.
\\ Choosing $r=h^{1+\gamma/2}=o(h)$, where $0<\gamma<2$ is a positive parameter, we obtain  
\begin{equation}
\begin{array}{l}
\displaystyle\int_{Q_{h}^{\alpha}\setminus\hat{Q}_{h}^{\alpha}}|\nabla v^{\alpha}|^{2}dx=o(h^{n})
 \end{array}
\label{eq:8}
\end{equation}

Note that Friedrich's inequality gives 
\begin{equation}
\begin{array}{l}
\displaystyle\int_{Q_{h}^{\alpha}\setminus\hat{Q}_{h}^{\alpha}}|v^{\alpha}-1 |^{2}dx\leq C(r)\int_{Q_{h}^{\alpha}\setminus\hat{Q}_{h}^{\alpha}}|\nabla v^{\alpha}|^{2}dx=o(h^{n+1})
 \end{array}
\label{eq:9}
\end{equation}

Differentiating $(5.6)$, we have
\begin{equation}
\begin{array}{l}
\displaystyle\frac{\partial w_{h}^{\varepsilon}}{\partial x_{i}}=\frac{\partial w}{\partial x_{i}}+\sum_{\alpha}\frac{\partial w}{\partial x_{i}}(v^{\alpha}-1)\phi_{\alpha}+\sum_{\alpha}\frac{\partial v^{\alpha}}{\partial x_{i}}w\phi_{\alpha}+\sum_{\alpha}\frac{\partial \phi_{\alpha}}{\partial x_{i}}(v^{\alpha}-1)w
 \end{array}
\label{eq:10}
\end{equation}
We substitute $(6.10)$ into $\Gamma^{\varepsilon}[w_{h}^{\varepsilon}]$ to obtain
\begin{align*}\Gamma^{\varepsilon}[w_{h}^{\varepsilon}]&=\int_{G^{\varepsilon}(\omega)}|\nabla w_{h}^{\varepsilon}|^{2}+\lambda |w_{h}^{\varepsilon}
|^{2}+2fw_{h}^{\varepsilon}dx
=\int_{G^{\varepsilon}(\omega)}|\nabla w|^{2}+
\lambda v|w|^{2}+\\&+2fwdx+\sum_{\alpha=1}^{n(h)}\int_{Q_{h}^{\alpha}}|\nabla v^{\alpha}|^{2}w^{2}\phi_{\alpha}^{2}dx+\sum_{i=1}^{5}L_{i}(\varepsilon,h).
\end{align*}
where
\[L_{1}(\varepsilon,h)=\sum_{\alpha}^{N(h)}2\int_{Q_{h}^{\alpha}}(f+\lambda w)(v^{\alpha}-1)w\phi_{\alpha}dx\]
\begin{align*}L_{2}(\varepsilon,h)&=\\&\sum_{\alpha,\beta}^{N(h)}2\int_{Q_{h}^{\alpha}\cap Q_{h}^{\beta}}\{\sum_{i=1}^{n}\left(\frac{\partial w}{\partial x_{i}}\phi_{\alpha}+\frac{\partial\phi_{\alpha}}{\partial x_{i}}w\right)\left(\frac{\partial w}{\partial x_{i}}\phi_{\beta}+\frac{\partial\phi_{\beta}}{\partial x_{i}}w\right)\\&+\lambda w^{2}\phi_{\alpha}\phi_{\beta}\}(v^{\alpha}-1)(v^{\beta}-1)dx\end{align*}
\[L_{3}(\varepsilon,h)=\sum_{\alpha,\beta}^{N(h)}2\int_{Q_{h}^{\alpha}\cap Q_{h}^{\beta}}\sum_{i=1}^{n}\left(\frac{\partial w}{\partial x_{i}}\phi_{\alpha}+\frac{\partial\phi_{\alpha}}{\partial x_{i}}w\right)\left(\frac{\partial w}{\partial x_{i}}\phi_{\beta}+\frac{\partial\phi_{\beta}}{\partial x_{i}}w\right)(v^{\beta}-1)dx\]
\[L_{4}(\varepsilon,h)=\sum_{\alpha,\beta}^{N(h)}\sum_{i=1}^{n} 2\int_{Q_{h}^{\alpha}\cap Q_{h}^{\beta}}\frac{\partial v^{\alpha}_{h}}{\partial x_{i}}\frac{\partial v^{\beta}_{h}}{\partial x_{i}}\phi_{\alpha}\phi_{\beta}w^{2}dx\]
\[L_{5}(\varepsilon,h)=\sum_{\alpha,\beta}^{N(h)}\sum_{i=1}^{n} 2\int_{Q_{h}^{\alpha}}\frac{\partial w}{\partial x_{i}}w\phi_{\alpha}\frac{\partial (v^{\alpha}_{h}-1)}{\partial x_{i}}dx\]

Taking into account the properties of $\phi_{\alpha}$, $(6.7)-(6.9)$ and the fact that the number of cubes $Q_{h}^{\beta}$ which intersect $Q_{h}^{\alpha}$ is no more than $3^{n}$, we have
\[\lim_{h\rightarrow 0}\lim_{\varepsilon\rightarrow 0}\sum_{i=1}^{5}L_{i}(\varepsilon,h)=0\]
From the properties of smooth functions $\phi_{\alpha}$ we now have
\[\int_{Q_{h}^{\alpha}}w^{2}\phi_{\alpha}|\nabla v^{\alpha}|^{2}dx\leq C\bar{w}_{\alpha}^{2}\text{cap}(x^{\alpha},h,\varepsilon,\omega)\] where $\bar{w}_{\alpha}$ is the mean of $w_{\alpha}$ over $Q_{h}^{\alpha}$. Summing over the cubes and letting $\varepsilon$ tend to zero we have
\begin{equation}
\begin{array}{l}
\limsup_{\varepsilon\rightarrow 0}\sum_{\alpha}\int_{Q_{h}^{\alpha}}|\nabla v^{\alpha}|^{2}w^{2}\phi_{\alpha}^{2}dx\leq\int_{D}cw^{2}dx+O(h)
 \end{array}
\label{eq:11}
\end{equation}
Combine these inequalities to see that $\displaystyle\limsup_{\varepsilon\rightarrow 0} \Gamma^{\varepsilon}[u^{\varepsilon}]\leq\bar{\Gamma}[w]$ for all twice differentiable functions with compact support in $D$. Using a density argument, this inequality holds for all $w\in H_{0}^{1}(D)$.

Step 2: To show the reverse inequality $\displaystyle\liminf_{\varepsilon\rightarrow 0}\Gamma^{\varepsilon}[u^{\varepsilon}]\geq\bar{\Gamma}[u]$, pick a sequence $u_{\delta}(x)\in C_{0}^{1}(D)$ such that $\displaystyle\|u_{\delta}-u\|_{H_{0}^{1}(D)}\leq\varepsilon$, where $u$ is the weak limit of the sequence of minimizers $u^{\varepsilon}$ of $\Gamma^{\varepsilon}[\cdot]$ in $H_{0}^{1}(D)$.
\\ According to lemma $3.2$ on \cite{Khrus05} pg.73, there is a sequence $\{u_{\delta}^{\varepsilon}\}\in H_{0}^{1}(D,F^{\varepsilon}(\omega))=\{v\in H_{0}^{1}(D):v=0\text{ in }F^{\varepsilon}(\omega)\}$ that converges to $u_{\delta}$ and satisfies $\displaystyle\|u_{\delta}^{\varepsilon}-u^{\varepsilon}\|_{H^{1}(D)}\leq C\|u_{\delta}-u\|_{H^{1}(D)}$.

Take now the cubes $Q_{h}^{\alpha}$ that belong to the set $D_{\tilde{\delta}}=\{x\in D:|u_{\delta}(x)|\geq\tilde{\delta}\}$ for positive parameter $\tilde{\delta}$ and in each of these cubes define the function $\displaystyle v_{\alpha}^{\varepsilon}=\frac{u_{\delta}^{\varepsilon}}{u^{\varepsilon}}$ so that $\displaystyle v_{\alpha}^{\varepsilon}\rightarrow 1$ weakly in $L^{2}(Q_{h}^{\alpha})$. 

Clearly,
\begin{equation}
\begin{array}{l}
\int_{Q_{h}^{\alpha}}|\nabla v_{\alpha}^{\varepsilon}|^{2}+h^{-2-\gamma}|v_{\alpha}^{\varepsilon}-1|^{2}dx\geq\text{cap}(x^{\alpha},\varepsilon,h,\gamma,\omega)
 \end{array}
\label{eq:12}
\end{equation}
and
\begin{equation}
\begin{array}{l}
 \frac{\partial v_{\alpha}^{\varepsilon}}{\partial x_{i}}=\frac{1}{u_{\delta}}\frac{\partial u_{\delta}^{\varepsilon}}{\partial x_{i}}-\frac{1}{u_{\delta}}\frac{\partial u_{\delta}}{\partial x_{i}}-\frac{u_{\delta}^{\varepsilon}-u_{\delta}}{u_{\delta}^{2}}\frac{\partial u_{\delta}}{\partial x_{i}}
  \end{array}
\label{eq:13}
\end{equation}
Using the expansion $(6.13)$ in $(6.12)$ and taking into account that $u_{\delta}^{\varepsilon}\rightarrow u_{\delta}$ strongly, we get
\begin{equation}
\begin{array}{l}
 \int_{Q_{h}^{\alpha}}\left|\nabla u_{\delta}^{\varepsilon}\right|^{2}dx\geq\text{cap}(x^{\alpha},\varepsilon,h,\omega)[\min_{Q_{h}^{\alpha}}|u_{\delta}|]^{2}+\frac{[\min_{Q_{h}^{\alpha}}|u_{\delta}|]^{2}}{[\max_{Q_{h}^{\alpha}}|u_{\delta}|]^{2}}\int_{Q_{h}^{\alpha}}|\nabla u_{\delta}|^{2}dx-G(\varepsilon,\delta,\tilde{\delta},h)
  \end{array}
\label{eq:14}
\end{equation}
where $\displaystyle\lim_{\varepsilon\rightarrow 0}G(\varepsilon,\delta,\tilde{\delta},h)=0$ for fixed $h,\delta,\tilde{\delta}$.

Finally, we sum over all cubes that intersect $G^{\varepsilon}(\omega)$, 
\begin{align*}
\Gamma^{\varepsilon}[u_{\delta}^{\varepsilon}]
\geq\sum_{\alpha=1}^{N}&\int_{Q_{h}^{\alpha}}|\nabla u_{\delta}|^{2}+\sum_{\alpha=1}^{N}\frac{\text{cap}(x^{\alpha},\varepsilon,h,\omega)}{h^{n}}[\sup_{Q_{h}^{\alpha}}|u_{\delta}|]^{2}h^{n}+\int_{D}\lambda|u_{\delta}^{\varepsilon}|^{2}+2fu_{\delta}^{\varepsilon}\\&  -\sum_{i=1}^{n}\frac{[\sup_{Q_{h}^{\alpha}}|u_{\delta}|]^{2}-[\min_{Q_{h}^{\alpha}}|u_{\delta}|]^{2}}{[\sup_{Q_{h}^{\alpha}}|u_{\delta}|]^{2}}\int_{Q_{h}^{\alpha}}|\nabla u_{\varepsilon}|^{2} dx-NG(\varepsilon,\delta,\tilde{\delta},h)
\end{align*}

We let $h\rightarrow 0$ for fixed $\delta$ to obtain
\[\lim_{\varepsilon\rightarrow 0}\Gamma^{\varepsilon}[u_{\delta}^{\varepsilon}]\geq\int_{D_{\tilde{\delta}}}|\nabla u_{\delta}|^{2}+cu_{\delta}^{2}dx+\int_{D}\lambda u_{\delta}^{2}+2fu_{\delta}^{2}dx\].

Let now $\delta,\tilde{\delta},\varepsilon$ tend to zero: $\displaystyle\bar{\Gamma}[u]\leq\bar{\Gamma}[v]$ for all $v\in H_{0}^{1}(D)$.
\end{proof}
\section{The Neumann problem, strong and weak connectivity}
In \cite{Khrus05}, the Neumann problem for elliptic equations is considered. The conditions of convergence are formulated in terms of mean local characteristics, that we will describe. 

Let $D\subset\mathbb{R}^{n}$ be a bounded domain and $G^{\varepsilon}(\omega)\cap D$ be the domain constructed in section $5$. We seek the compactness of the sequence $u^{\varepsilon}(x)$ defined in $G^{\varepsilon}(\omega)\cap D$, i.e, the possibility that a subsequence $u^{\varepsilon_{k}}(x)$ converging to some $u\in L^{2}(D)$ in the sense that \[\|u^{\varepsilon_{k}}-u\|_{L^{2}(G^{\varepsilon}(\omega)\cap D)}\rightarrow 0\] as $\varepsilon\rightarrow 0$.

The domains $\{G^{\varepsilon}(\omega)\}_{\varepsilon>0}$ are said to be strongly connected if any sequence of function $u^{\varepsilon}$ defined in $G^{\varepsilon}(\omega)$ is compact with respect to the last convergence. 

The main local characteristic of the medium is the functional
	\[ P_{\varepsilon,h}^{z}(\xi)=\inf_{v^{\varepsilon}\in W^{1,2}(Q_{h}^{z}\cap F^{\varepsilon}(\omega))}\int \{|\nabla v^{\varepsilon}|^{2}+h^{-2-\gamma}|v^{\varepsilon}-(x-z,\xi)|^{2}\}dx\]
\\ Note that the function $v_{\xi}^{\varepsilon}$ which minimizes $\displaystyle P_{\varepsilon,h}^{z}(\xi)$ for any $\xi\in \mathbb{R}^{n}$ can be written in the form $\displaystyle v_{\xi}^{\varepsilon}=\sum_{i=1}^{n}\xi_{i}v_{i}^{\varepsilon}$, where $v_{i}^{\varepsilon}$ is the corresponding minimizer for $\xi_{i}=e_{i}$.
\\ Thus, if we write
\begin{align*} a_{ij}(z,\varepsilon,h,\omega)=&\int\nabla v_{i}^{\varepsilon}\cdot\nabla v_{j}^{\varepsilon}
\\ &+h^{-2-\gamma}[v_{i}^{\varepsilon}-(x_{i}-\xi_{i})][v_{j}^{\varepsilon}-(x_{j}-\xi_{j})]dx
\end{align*}
then $P_{\varepsilon,h}^{z}(\xi)$ takes the quadratic form $\displaystyle P_{\varepsilon,h}^{z}(\xi)=\sum_{i,j=1}^{n}a_{ij}(z,\varepsilon,h,\omega)\xi_{i}\xi_{j}$.
\\ The matrix $[a_{ij}]$ is the local mean conductivity tensor of the medium at the point $z$.

Similar capacity-type functionals can be defined for weakly connected domains. Such domains consist of a finite number of strongly connected components. We note that the domains $G^{\varepsilon}(\omega)$, $\varepsilon>0$ are weakly connected domains, since the connectivity function $g(|\cdot|)$ gives, in general, only connected components.

Furthermore, the temperature distribution $u^{\varepsilon}$ of the porous medium is the minimizer of \[ J_{G^{\varepsilon}(\omega)}[u^{\varepsilon}]=\int_{G^{\varepsilon}(\omega)}|\nabla u^{\varepsilon}|^{2}dx\] over the class $\displaystyle\{ u^{\varepsilon}\in W^{1,2}(G^{\varepsilon}(\omega)):u^{\varepsilon}=u_{0}\text{ on }\partial(D)\}$.
\\ Suppose for the moment that $u^{\varepsilon}$ can be extended to $\tilde{u}^{\varepsilon}$ such that \\ $\displaystyle ||\tilde{u}^{\varepsilon}||_{W^{1,2}(D)}\leq C$ uniformly with respect to $\varepsilon$ (see \cite{BoSu07} for improvements). Then, up to a subsequence $\tilde{u}^{\varepsilon}$ converges to a function $u$ in $L^{2}(D)$, which is almost linear to every sufficiently small cube $K_{h}^{z}$, i.e. $u(x)=u(z)+(x-z,\nabla u(z))+o(h^{2})$. Then for $\varepsilon$ small enough,
\[\int_{Q_{h}^{z}\cap G^{\varepsilon}(\omega)}|v^{\varepsilon}(x)-(x-z,\nabla u(z))|^{2}dx=O(h^{n+4})\]
where $v^{\varepsilon}(x)=u^{\varepsilon}(x)-u(z)$. Thus, we can assume that the function $v^{\varepsilon}(x)$ is the minimizer of \[ J_{h}^{\varepsilon}(v^{\varepsilon})=\int_{Q_{h}^{z}\cap G^{\varepsilon}(\omega)}|\nabla v^{\varepsilon}|^{2}dx=O(h^{n})\] 
Under this consideration, \[J_{h}^{\varepsilon}(v^{\varepsilon})- P_{\varepsilon,h}^{z}(\nabla u(z))=o(h^{n})\] as $h\rightarrow 0$. 

The construction of $G^{\varepsilon}(\omega)$ says that $J_{h}^{\varepsilon}$ is periodic in law and independent at large distances in the following sense:

Define the function $\displaystyle a(\omega,z)=1-\min\{1,X_{T_{\rho}(l_{ij})}\}$ which is $0$ in the union of random channels with endpoints the points of the Poisson process and $1$ elsewhere. Consider the random functional $\displaystyle J(\omega)(u,A)=\int_{A}a(\omega,x)|\nabla u|^{2}dx=\int_{G(\omega)}|\nabla u|^{2}dx$ and define $\rho_{\varepsilon}J=J^{\varepsilon}(\omega)$.

Thus, theorem $2$ shows that $\displaystyle\lim_{h\rightarrow 0}\frac{v^{\varepsilon}}{|Q_{h}^{z}|}$ exists. This implies that \[\lim_{h\rightarrow 0}\lim_{\varepsilon\rightarrow 0}\frac{a_{ij}(z,\varepsilon,h,\omega)}{h^{3}}=a_{ij}(x)=\bar{a}\] exists almost all $\omega\in\Omega_{1}$ and it is constant.

It is worth to mention that, in order to apply the same homogenization technique, a major open problem is to prove the extension of functions from randomly perforated domains for Neumann problems, since the strong connectivity assumption of \cite{Khrus05} does not hold with probability one.

\section{Homogenization for sets of decreasing volume}

In \cite{Khrus05}, \cite{Pankr03} the homogenization problem for sets of decreasing volume is considered. In such problems, we allow the measure of the domains of definition of functions to vanish. A main assumption to define appropriate type of convergence is a uniform density of the vanishing set within the domain $D$: If $F^{\varepsilon}\subset D$ is such that $|F^{\varepsilon}|\rightarrow 0$ as $\varepsilon\rightarrow 0$, then for sufficiently small $\varepsilon>0$ we require
\[Cr^{n}|F^{\varepsilon}|\leq|B(x,r)\cap F^{\varepsilon}|\leq C^{-1}r^{n}|F^{\varepsilon}|\]
for every ball $B(x,r)$ centered at $x\in D$ of radius $r>0$. 

In particular, let us consider the set $F^{\varepsilon}(\omega)$ of tubes as in section $5$, so that $|F^{\varepsilon}(\omega)|\rightarrow 0$ as $\varepsilon\rightarrow 0$. We want to show that the density assumption holds for these domains with probability $1$. Recall that $F^{\varepsilon}(\omega)=\varepsilon F(\omega)$ is the union of tubes with endpoints the points of $X$ which has the same distribution with the model of density $\lambda/\varepsilon$ and height at least $c_{1}\varepsilon$ and at most $c_{2}\varepsilon$. Thus, as $\varepsilon$ decreases, the mean density increases and the distance between the pairs of connected points decreases.

Suppose that the last density condition does not hold. This means that there is a ball $B(x,r)$ such that for all $\varepsilon<\bar{\varepsilon}(\omega)$ and $\omega\in\Omega_{2}\subset\Omega_{1}$ with $|\Omega_{2}|\neq 0$, 
\[|B(x,r)\cap F^{\varepsilon}(\omega)|=0\]
This means that, for all small $\varepsilon$, there are no points $x_{i},x_{j}\in B(x,r)$ of the process $X$ with $c_{1}\varepsilon\leq|x_{i}-x_{j}|\leq c_{2}\varepsilon$. On the other hand,
\[P(X(B(x,r))=0)=e^{-\lambda r^{n}/\varepsilon}\rightarrow 0\] as $\varepsilon\rightarrow 0$ which is a contradiction.

\section{Further remarks}
The modeling based on connectivity functions may be used for more general shapes and random structures, since the total measure of the sets that it produces is forced to tend to zero.

\bibliographystyle{plain}
\bibliography{hcp1}

\end{document}